\documentclass[a4paper, 12pt, oneside, reqno, notitlepage]{amsart}
\usepackage[margin=3cm]{geometry}
\usepackage{amsmath,amssymb,amsthm,graphicx,mathrsfs,bbm,url}
\usepackage{amsthm}
\usepackage{wrapfig}
\usepackage{enumitem}
\usepackage{mathtools}
\usepackage{comment}
\usepackage[utf8]{inputenc}
\usepackage[T1]{fontenc}
\usepackage[usenames,dvipsnames]{color}
\usepackage[colorlinks=true,linkcolor=Red,citecolor=Green]{hyperref}
\usepackage[super]{nth}
\usepackage[open, openlevel=2, depth=3, atend]{bookmark}
\hypersetup{pdfstartview=XYZ}
\usepackage[font=footnotesize]{caption}
\usepackage{a4wide}
\usepackage[all]{xy}
\usepackage{epstopdf}
\usepackage{hyperref}

\theoremstyle{plain}
\newtheorem{theorem}{Theorem}[section]
\newtheorem*{theorem*}{Theorem}

\newtheorem{lemma}[theorem]{Lemma}
\newtheorem{proposition}[theorem]{Proposition}
\newtheorem{corollary}[theorem]{Corollary}

\theoremstyle{definition}

\newtheorem{example}[theorem]{Example}

\theoremstyle{remark}
\newtheorem{remark}[theorem]{Remark}

\numberwithin{equation}{section}

\newcommand{\C}{\mathbb{C}}
\newcommand{\R}{\mathbb{R}}

\newcommand{\E}{\mathcal{E}}

\newcommand{\V}{\mathbb{V}}

\newcommand{\X}{\mathbf{X}}

\newcommand{\HH}{\mathbb{H}}

\newcommand{\mc}{\mathcal}

\newcommand{\dd}{\mathrm{d}}
\newcommand{\e}{\mathbf{e}}

\newcommand{\End}{\mathrm{End}}

\DeclareMathOperator{\id}{id}

\newcommand{\be}{\begin{equation}}
\newcommand{\ee}{\end{equation}}

\title[The Pestov identity on the frame bundle]{The Pestov identity on the frame bundle and associated homogeneous fibrations}

\author[M. Ceki\'c]{Mihajlo Ceki\'c}
\address{Université Paris-Est Créteil, CNRS, Laboratoire d’analyse et de mathématiques appliquées, 94010, Créteil, France}
\email{mihajlo.cekic@cnrs.fr}

\author[T. Lefeuvre]{Thibault Lefeuvre}
\address{Laboratoire de Mathématiques d’Orsay, Université Paris-Saclay, Bâtiment 307, 91405 Orsay Cedex France}
\email{thibault.lefeuvre1@universite-paris-saclay.fr}

\author[A. Moroianu]{Andrei Moroianu}
\address{Université Paris-Saclay, CNRS,  Laboratoire de mathématiques d'Orsay, 91405, Orsay, France, 
and Institute of Mathematics “Simion Stoilow” of the Romanian Academy, 21 Calea Grivitei, 010702 Bucharest, Romania}
\email{andrei.moroianu@math.cnrs.fr}

\author[U. Semmelmann]{Uwe Semmelmann}
\address{Institut f\"ur Geometrie und Topologie, Fachbereich Mathematik, Universit{\"a}t Stuttgart, Pfaffenwaldring 57, 70569 Stuttgart, Germany}
\email{uwe.semmelmann@mathematik.uni-stuttgart.de}

\begin{document}

\begin{abstract}
In this short note, we prove a global Pestov identity on the (orthonormal) frame bundle of a Riemannian manifold and deduce  similar identities on associated homogeneous fibrations. As a particular example, this provides a concise proof of the Pestov identity on the unit tangent bundle of the manifold.
\end{abstract}

\maketitle

\section{Introduction}

\subsection{Context} The Pestov formula is a powerful energy identity that plays a central role in geometric inverse problems and dynamical rigidity results, particularly on Riemannian manifolds with Anosov geodesic flow \cite{Guillemin-Kazhdan-80, Croke-Sharafutdinov-98, Dairbekov-Sharafutdinov-03, Pestov-Uhlmann-05,Guillarmou-Lefeuvre-18,Cekic-Lefeuvre-Moroianu-Semmelmann-24, Paternain-Salo-Uhlmann-13}. It was first discovered by Mukhometov \cite{Mukhometov-75,Mukhometov-81} and Amirov \cite{Amirov-86}, before being generalized by Pestov and Sharafutdinov \cite{Pestov-Sharafutdinov-88,Sharafutdinov-94}. Later, Knieper \cite{Knieper-02} found an intrinsic formulation. A twisted identity involving an auxiliary vector bundle was also recently obtained by Guillarmou, Paternain, Salo and Uhlmann \cite{Paternain-Salo-Uhlmann-15, Guillarmou-Paternain-Salo-Uhlmann-16}. The usefulness of the Pestov formula comes from the way it links dynamical information (involving the geodesic vector field) and geometric information (curvature, vertical derivatives, etc.). It is also a key ingredient in the proof of the ergodicity of the frame flow \cite{Cekic-Lefeuvre-Moroianu-Semmelmann-22,Cekic-Lefeuvre-Moroianu-Semmelmann-24,Cekic-Lefeuvre-Moroianu-Semmelmann-24-2}. When specialized to functions of a given Fourier degree, the Pestov formula turns out to be equivalent to the Weitzenböck formula on symmetric tensors \cite{Cekic-Lefeuvre-Moroianu-Semmelmann-25}. See also \cite{Paternain-Salo-Uhlmann-23, Lefeuvre-book} for textbook treatments of the Pestov identity and for more references.

\subsection{Content of this note} The purpose of this short note is to establish the Pestov identity on the frame bundle of a Riemannian manifold (Theorem \ref{theorem:global-pestov-frame}), and to show how to derive from it all the known Pestov identities on the unit tangent bundle (Corollary \ref{corollary:associated}). As a remark, we observe that the idea of establishing a Pestov identity on the frame bundle was also considered in \cite{Egidi-16}, although our formula seems to differ slightly from the one in loc. cit.

The note is organized as follows:
\begin{itemize}
\item In \S\ref{section:frame-bundle}, we recall some standard facts about the geometry of the orthonormal frame bundle of a Riemannian manifold.
\item In \S\ref{section:pestov}, we establish the universal Pestov formula on the frame bundle (Theorem \ref{theorem:global-pestov-frame}) and show how to derive a corresponding Pestov identity on any homogeneous fibration associated to the frame bundle (Corollary \ref{corollary:associated}). As a particular example, we recover the usual Pestov identity on the unit tangent bundle.
\item In \S\ref{section:applications}, we derive some applications of this formula on negatively curved manifolds.
\end{itemize}

The usual proofs of the Pestov identity are quite involved (see \cite{Paternain-Salo-Uhlmann-15} or \cite[Chapter 14]{Lefeuvre-book} for instance). However, the proof provided in this paper is of a more conceptual nature, and we believe that it could easily be generalized to other contexts (such as magnetic or thermostat flows), where such an identity does not seem to be available yet. \\

\noindent \textbf{Acknowledgement:} The authors wish to thank the CIRM, where part of this article was written, for support and hospitality. TL is supported by the European Research Council (ERC) under the European Union’s Horizon 2020 research and innovation programme (Grant agreement no. 101162990 — ADG). AM is partially supported by the PNRR-III-C9-2023-I8 grant CF 149/31.07.2023 {\em Conformal Aspects of Geometry and Dynamics}. MC received funding from an Ambizione grant (project number 201806) from the Swiss National Science Foundation.

\section{Frame bundle geometry and structural equations}

\label{section:frame-bundle}

\subsection{Frame bundle geometry} 

\label{ssection:frame-bundle-geometry}

Let $(M,g)$ be a smooth closed connected oriented $n$-dimensional Riemannian manifold with $n \geq 2$. Let $$\pi_{SM} : SM = \{v \in TM \mid g(v, v) = 1\} \to M$$ be the unit tangent bundle, and let $\pi_{FM} : FM \to M$ the principal $\mathrm{SO}(n)$-bundle of oriented orthonormal frames over $M$. A frame $w \in FM$ is an isometry $w : (\R^n,g_{\mathrm{Euc}}) \to (T_xM,g_x)$, where $x := \pi_{FM}(w)$ and $g_{\mathrm{Euc}}$ denotes the Euclidean metric on $\R^n$. The right action $R_a$ of $a \in \mathrm{SO}(n)$ on $FM$ is then given by composition $R_a w = w \circ a$. Let $(\e_1,\dotsc,\e_n)$ denote the canonical basis of $\R^n$.
Note that any frame $w \in FM$ also induces an isometry $w : \Lambda^2 \R^n \to \Lambda^2 T_xM$, where $\Lambda^2 \R^n$ is equipped with the metric making $(\e_i \wedge \e_j)_{1 \leq i < j \leq n}$ into an orthonormal basis. We also emphasize that throughout the paper $1$-forms are identified with vectors via the metric, and $\Lambda^2\R^n$ is identified with skew-symmetric endomorphisms $\mathfrak{so}(n)$ via $\xi \mapsto \left(\theta \mapsto \iota_\theta \xi\right)$, where $\iota_\theta$ denotes contraction with the vector $\theta\in\R^n$. 

The vertical bundle $\V \subset T(FM)$ is defined as
\begin{equation}
\label{equation:vertical}
\V := \ker \dd\pi_{FM}.
\end{equation}
Given a frame $w \in FM$ and $2$-form $\xi \in \Lambda^2 \R^n$, there is a corresponding vertical tangent vector $\xi^{\V} \in \V_w$ defined as
\begin{equation}
\label{equation:vertical-lift}
\xi^{\V} =  \dfrac{\dd}{\dd t}|_{t=0} (w \circ e^{t\xi}) \in \V_w.
\end{equation}
The \emph{fundamental} vector fields are the vertical vector fields $Y_\xi(w) := \xi^\V$, where $\xi \in \Lambda^2 \R^n$. Given $(\omega_\alpha)_{\alpha}$, a basis of $\Lambda^2 \R^n$, this yields a family $(Y_{\omega_\alpha})_{\alpha}$ of fundamental vertical vector fields $Y_{\omega_\alpha} \in C^\infty(FM,\V)$ spanning $\V$ at every point $w \in FM$ of the frame bundle. 

We let $\HH \subset T(FM)$ be the horizontal bundle induced by the Levi-Civita connection. Note that $T(FM)$ splits into $T(FM) = \HH \oplus \V$. For any $w \in FM$, $\dd\pi_{FM}(w) : \HH_w \to T_{\pi_{FM}(w)}M$ is an isomorphism; given $Z \in T_{\pi_{FM}(w)}M$, we denote by $Z^{\HH} \in \HH_w$ its \emph{horizontal lift} as the unique horizontal vector such that $\dd\pi_{FM}(w)Z^{\HH} = Z$. 
For every $\theta\in\R^n$, the \emph{standard} vector field associated to $\theta$ is the horizontal vector field defined as $B_\theta(w) := w(\theta)^{\HH} \in \HH_w$. Similarly as for $\V$, denoting by $(\e_1,\ldots,\e_n)$ the standard orthonormal basis of $\R^n$, the family of standard vector fields $(B_{\e_i})_{1 \leq i \leq n}$ trivializes $\HH$ over $FM$.

The Riemann curvature tensor of $g$ is initially defined as the $(1,3)$-tensor $\mc{R}(X,Y) := \nabla_X\nabla_Y - \nabla_Y\nabla_X-\nabla_{[X,Y]}$, where $\nabla$ denotes the Levi-Civita connection of $g$; it can be identified with a symmetric endomorphism of $\Lambda^2 TM$ using its symmetries. Namely
\begin{equation}
    \label{equation:utile}
    \langle \mc{R}(X \wedge Y), W \wedge Z\rangle := \langle\mc{R}(X,Y)W, Z\rangle,
\end{equation}
where the scalar product on the left is the one induced by the metric on $2$-forms. With this convention, $\mc{R}$ is equal to the identity $\mathrm{Id}_{\Lambda^2 TM}$ when $(M,g)$ is the real hyperbolic space.

\begin{lemma}
For all $\xi,\xi' \in \Lambda^2 \R^n$, $\theta,\theta' \in \R^n$ and $w \in FM$ we have:
\begin{equation}
\label{equation:commutateur}
 [Y_\xi,Y_{\xi'}]= Y_{[\xi,\xi']}, \quad [Y_\xi,B_\theta ] = B_{\xi\theta}, \quad [B_{\theta},B_{\theta'}] = - (w^{-1}\mc{R}_{\pi_{FM}(w)}(w(\theta) \wedge w(\theta')))^{\V},
\end{equation}
where $\xi\theta$ is the vector obtained by applying $\xi$ (as an endomorphism of $\R^n$) to $\theta$.
\end{lemma}

\begin{proof}These identities are standard, see \cite[Chapter 3, Section 5]{Kobayashi-Nomizu-96} or \cite[Equation (15)]{Cleyton-Moroianu-Semmelmann-21} for the first identity in \eqref{equation:commutateur}, \cite[Equation (16)]{Cleyton-Moroianu-Semmelmann-21} for the second, and \cite[Equation (24)]{Cleyton-Moroianu-Semmelmann-21} for the third. \end{proof}

Finally, if $(\e_i)_{1 \leq i \leq n}$ denotes as before the standard basis of $\R^n$, and $(\omega_{\alpha})_{\alpha}$ is an orthonormal frame of $\Lambda^2 \R^n$, we can equip $FM$ with the Riemannian metric $G_{FM}$ turning $(B_{\e_i}, Y_{\omega_{\alpha}})_{1 \leq i \leq n, 1\leq\alpha\leq\frac{n(n-1)}2}$ into an orthonormal basis. The projection $\pi_{FM} : (FM,G_{FM}) \to (M,g)$ is then a Riemannian submersion. The \emph{Liouville measure} $\mu$ on $FM$ is defined as the Riemannian volume element associated with the metric $G_{FM}$.

\subsection{Unit tangent bundle geometry}

\label{ssection:unit-tangent-bundle}

The tangent bundle of $SM$, the unit tangent bundle, splits as
\[
T(SM) = \HH \oplus \V
\]
where $\V: = \ker \dd \pi_{SM}$ and $\HH$ is the horizontal bundle induced by the Levi-Civita connection. When the context is clear, we use the same letters for the vertical and horizontal bundles on $SM$ and $FM$; later, we will add a subscript to avoid confusion. The map $\dd \pi_{SM} : \HH_{(x,v)} \to T_xM$ is an isomorphism.

We define the \emph{connection map} $\mathcal{K}_{(x,v)} : T_{(x,v)}(SM) \to T_xM$ as
\[
\mathcal{K}_{(x,v)}(\zeta) :=  \dfrac{\dd}{\dd t}|_{t=0}(\tau_{x(t) \to x(0)} v(t)) \in v^\perp,
\]
for every $\zeta \in T_{(x,v)}(SM)$, where $t \mapsto z(t) = (x(t),v(t)) \in SM$ is a path such that $\dot{z}(0)=\zeta$ and $\tau_{x(t)\to x(0)}$ denotes the parallel transport along the path $[0,t] \ni s \mapsto x(t-s)$ with respect to the Levi-Civita connection. The map $\mathcal{K}_{(x,v)} : \V_{(x,v)} \to v^\perp \subset T_xM$ is an isomorphism (given by the identification of $\V_{(x, v)}$ with the tangent space to the sphere $S_xM \subset T_xM$ at $v$ which is further naturally identified with $v^\perp$) and $\ker \mathcal{K}_{(x,v)}=\HH_{(x,v)}$. See \cite[Lemma 1.13]{Paternain-99} for further details. 

\subsection{Geodesic and frame flows}

The geodesic flow $(\varphi_t)_{t \in \R}$ on $SM$ is defined as follows: given $v \in SM$, $x := \pi_{SM}(v)$,  we let $\R \ni t \mapsto \gamma_{x,v}(t) \in M$ be the geodesic generated by $\gamma_{x,v}(0)=x$, $\dot{\gamma}_{x,v}(0)=v$; then
\begin{equation}
\label{equation:geodesic-flow}
\varphi_t(v) := (\gamma_{x,v}(t),\dot{\gamma}_{x,v}(t)).
\end{equation}
We denote by $\X_{SM}$ the vector field generating the geodesic flow.

The frame flow $(\Phi_t)_{t \in \R}$ on the frame bundle $FM$ is then defined as follows: given an orthonormal frame $w=(v,e_2,\ldots,e_n)$, at $x := \pi_{FM}w$, we set
\begin{equation}
\label{equation:frame-flow}
\Phi_t(w) := (\varphi_t(v), \tau_{\gamma_{x,v}(t)} e_2, \ldots, \tau_{\gamma_{x,v}(t)} e_n) = (\gamma_{x, v}(t), \tau_{\gamma_{x, v}(t)}w),
\end{equation}
where $\tau_{\gamma_{x,v}(t)}$ denotes the parallel transport along the geodesic segment $(\gamma_{x,v}(s))_{s \in [0,t]}$ with respect to the Levi-Civita connection of $g$. We denote by $\X $ the vector field generating the frame flow. Note that, by definition, $\X  = B_{\e_1}$, where $\e_1$ is the first vector of the standard basis of $\R^n$. We will denote by the same symbol the corresponding operator $\X$ acting on vector-valued functions on $FM$.

\subsection{Structural equations}

The \emph{vertical gradient} is defined as
\begin{equation}
\label{equation:nablav}
\nabla_{\V} : C^\infty(FM) \to C^\infty(FM,\Lambda^2 \R^n), \qquad \nabla_{\V} f = \sum_{\alpha} Y_{\omega_\alpha} f~\omega_\alpha.
\end{equation}
The \emph{horizontal gradient} is then defined as the commutator of $\nabla_{\V}$ with $\X$, namely
\begin{equation}
\label{equation:nablah}
\begin{split}
\nabla_{\HH} : C^\infty(FM) \to C^\infty(FM,\Lambda^2 \R^n), \qquad \nabla_{\HH} := - [\X,\nabla_{\V}],
\end{split}
\end{equation}
We now wish to express explicitly \eqref{equation:nablah}: 

\begin{lemma}[Expression for the horizontal gradient]
\label{lemma:expression-nablah}
Let $(\e_1,\ldots,\e_n)$ be the standard basis of $\R^n$. Then:
\begin{equation}
\label{equation:expression-nablah}
\nabla_{\HH} = \sum_{j=2}^n B_{\e_j}~ \e_1 \wedge \e_j
\end{equation}
\end{lemma}
The expression \eqref{equation:expression-nablah} should be understood as follows: for all $f \in C^\infty(FM)$, $\nabla_{\HH} f = \sum_{j=2}^n B_{\e_j} f ~\e_1 \wedge \e_j$.

\begin{proof}
First, taking $\theta := \e_1$ in \eqref{equation:commutateur} and using $\X  = B_{\e_1}$, we get $[\X ,Y_\xi] = - B_{\xi \e_1}$. Now, if $(\omega_\alpha)_{\alpha}$ is a basis of $\Lambda^2 \R^n$, the sections $\omega_\alpha \in C^\infty(FM,\Lambda^2 \R^n)$ are constant, so we get from \eqref{equation:commutateur} and \eqref{equation:nablav}:
\[
\nabla_{\HH}f = -[\X,\nabla_{\V}] f = -\sum_{\alpha} [\X ,Y_{\omega_\alpha}] f~\omega_\alpha =\sum_{\alpha} B_{\omega_\alpha \e_1}f~\omega_\alpha.
\]
Taking the specific basis $\omega_{\alpha} = \e_i \wedge \e_j$ for $\alpha=(i,j)$ with $i < j$, we obtain the claimed result.
\end{proof}

We now compute the commutator $[\X,\nabla_{\HH}]$. For that purpose, we introduce the endomorphism $R_{FM} \in C^\infty(FM,\End(\Lambda^2\R^n))$ defined for $w \in FM $ and $\xi, \xi' \in \Lambda^2 \R^n$ by
\begin{equation}
\label{equation:curvature}
\langle  R_{FM}(w)\xi, \xi' \rangle  :=  \langle \xi, w^{-1}\mc{R}_{x}(w(\xi'\e_1) \wedge w(\e_1)) \rangle_{\Lambda^2 \R^n},
\end{equation}
where $x:=\pi_{FM}(w)$, and $\mc{R} \in C^\infty(M,S^2\Lambda^2TM)$ is the Riemann curvature tensor  (see \eqref{equation:utile}).

\begin{lemma}[Commutation relation for the horizontal gradient]
The following holds:
\begin{equation}
\label{equation:commutator-nablah}
[\X,\nabla_{\HH}] =  R_{FM} \nabla_{\V}.
\end{equation}
\end{lemma}

\begin{proof}
Let $(\e_1,\ldots,\e_n)$ be the standard basis of $\R^n$ and set $Y_{ij} := Y_{\e_i \wedge \e_j}$. Let $f \in C^\infty(FM)$, let $w \in FM$, and write $v := w(\e_1)$. Using \eqref{equation:curvature}, we have:
\[
\begin{split}
R_{FM}(w) \nabla_{\V} f & = R_{FM}(w)\left(\sum_{i< j} Y_{ij} f~\e_i \wedge \e_j \right) \\
& = \sum_{i <j}\sum_{i' < j'} \langle Y_{ij} f~ \e_i \wedge \e_j, w^{-1}\mc{R}_{\pi(w)}(w(\e_{i'} \wedge \e_{j'})v \wedge v)\rangle ~\e_{i'} \wedge \e_{j'}  \\
& = \sum_{i < j}\sum_{2 \leq j' \leq n} Y_{ij}f \langle  \e_i \wedge \e_j, w^{-1} \mc{R}_{\pi(w)}(w(\e_{j'}) \wedge v)\rangle ~\e_{1} \wedge \e_{j'}.
\end{split}
\]
On the other hand, we have:
\[
\begin{split}
[\X,\nabla_{\HH}] f(w) &= \sum_{j'=2}^{n} [\X ,B_{\e_{j'}}] f~\e_1 \wedge \e_{j'} = \sum_{j'=2}^{n} [B_{\e_1},B_{\e_{j'}}] f~\e_1 \wedge \e_{j'} \\
& = - \sum_{j'=2}^{n} (w^{-1}\mc{R}_{\pi_{FM}(w)}(v\wedge w(\e_{j'})))^{\V}f~ \e_1 \wedge \e_{j'} \\
& = - \sum_{i < j}\sum_{2 \leq j' \leq n} Y_{ij}f \langle \e_i \wedge \e_j, w^{-1}\mc{R}_{\pi_{FM}(w)}(v\wedge w(\e_{j'})) \rangle\e_1 \wedge \e_{j'} \\
& = R_{FM}(w) \nabla_{\V} f,
\end{split}
\]
where we used \eqref{equation:expression-nablah} in the first equality and \eqref{equation:commutateur} in the third. This concludes the proof.
\end{proof}

\begin{example}
\label{example:curvature}
Assume that $(M^n,g)$ is hyperbolic, i.e. of constant sectional curvature $-1$. Then for all $w \in FM$ and $\xi ,\xi' \in \Lambda^2 \R^n$, $\langle R_{FM}(w) \xi , \xi' \rangle_{\Lambda^2 \R^n} = - \langle \xi  \e_1, \xi'\e_1 \rangle$. Indeed, the Riemann curvature tensor acts as $\mc{R} = \mathrm{Id}_{\Lambda^2 TM}$ in constant curvature $-1$ and thus
\[
\begin{split}
\langle R_{FM}(w)\xi, \xi' \rangle 
& = \langle \xi, w^{-1}\mc{R}_{x}(w(\xi'\e_1) \wedge w(\e_1)) \rangle_{\Lambda^2 \R^n} \\
& = \langle w(\xi), w(\xi'\e_1) \wedge w(\e_1) \rangle = \langle \xi, \xi'\e_1 \wedge \e_1 \rangle = - \langle \xi\e_1, \xi'\e_1 \rangle.
\end{split} 
\]
\end{example}
\medskip

The formal adjoints of $\nabla_{\V}$ and $\nabla_{\HH}$ are respectively the first order differential operators 
\[
\nabla^*_{\V}, \, \nabla^*_{\HH}  : C^\infty(FM,\Lambda^2 \R^n) \to C^\infty(FM).
\]
Observe that for a general operator $H : C^\infty(FM) \to C^\infty(FM,\Lambda^2\R^n)$ of the form $Hf := \sum_\alpha H_\alpha f~\omega_\alpha$ (where $(\omega_\alpha)_{\alpha}$ is an orthonormal basis of $\Lambda^2 \R^n$), if the vector fields $H_\alpha \in C^\infty(FM,T(FM))$ preserve the Liouville measure $\mu$ on $FM$, then its formal adjoint is given by
\[
H^*(\sum_\alpha f_\alpha \omega_\alpha) = - \sum_{\alpha} H_\alpha f_\alpha.
\]
From this remark, and as $(Y_{ij})_{ij}$ and $(B_{\e_j})_j$ preserve $\mu$, we infer:
\begin{equation}
\label{equation:divergence}
\nabla_{\V}^*(\sum_{i < j} f_{ij}~\e_i \wedge \e_j) = - \sum_{i < j} Y_{ij}f_{ij}, \qquad \nabla^*_{\HH}(\sum_{i < j} f_{ij}~\e_i \wedge \e_j) = - \sum_{j=2}^n B_{\e_j}f_{1j}.
\end{equation}

The commutation relations for $\X,\nabla_{\V},\nabla_{\HH}$ then become (note that $\X^* = -\X$):
\begin{equation}
\label{equation:divergence-commutators}
[\X,\nabla^*_{\V}] = - \nabla^*_{\HH}, \qquad [\X,\nabla^*_{\HH}] = \nabla^*_{\V} R^*_{FM}.
\end{equation}

\begin{lemma}
\label{lemma:formule-chiante}
The following relation holds:
\begin{equation}
\label{equation:formule-chiante}
 \nabla^*_{\HH}\nabla_{\V} -\nabla_{\V}^* \nabla_{\HH} = -(n-1)\X .
 \end{equation}
\end{lemma}

\begin{proof}
We have:
\[
\begin{split}
\nabla^*_{\HH}\nabla_{\V} -\nabla_{\V}^* \nabla_{\HH} & = \sum_{j=2}^n [Y_{1j},B_{\e_j}]   = \sum_{j=2}^n B_{(\e_1 \wedge \e_j) \e_j}   = - \sum_{j=2}^n B_{\e_1} = -(n-1) \X  ,
\end{split}
\]
where the first equality follows from \eqref{equation:divergence}, the second from \eqref{equation:commutateur}, and the last one uses $\X  = B_{\e_1}$.
\end{proof}

\section{Pestov identities}

\label{section:pestov}

\subsection{Universal Pestov identity on the frame bundle}

In this paragraph, we prove the following identity, which we call \emph{universal Pestov identity} on the frame bundle. 

\begin{theorem}[Universal Pestov identity]
\label{theorem:global-pestov-frame}
For all $u \in C^\infty(FM)$,
\begin{equation}
\label{equation:global-pestov-frame}
\begin{split}
 \|\nabla_{\V} \X u\|^2_{L^2(FM, \Lambda^2 \mathbb{R}^n)} - \|\X \nabla_{\V} u\|^2_{L^2(FM, \Lambda^2\mathbb{R}^n)}&\\ 
 &\hspace{-4cm}= (n-1)\|\X u\|^2_{L^2(FM)} - \langle R_{FM} \nabla_{\V} u, \nabla_{\V} u \rangle_{L^2(FM, \Lambda^2\mathbb{R}^n)}.
\end{split}
\end{equation}
\end{theorem}

Observe that \eqref{equation:global-pestov-frame} only involves $L^2$-norms of second-order derivatives of the function, so it actually holds more generally for functions $u$ in the Sobolev space $H^2(FM)$. Theorem \ref{theorem:global-pestov-frame} actually follows after integration over $FM$ from the following pointwise Pestov identity:

\begin{proposition}
The following identity holds:
\begin{equation}
\label{equation:global-pestov-pointwise}
\X^*\nabla_{\V}^*\nabla_{\V}\X - \nabla_{\V}^*\X^*\X\nabla_{\V} = (n-1) \X^*\X - \nabla_{\V}^*R_{FM}\nabla_{\V}.
\end{equation}
\end{proposition}

\begin{proof}
We write:
\[
\begin{split}
\X^*\nabla_{\V}^*\nabla_{\V}\X - \nabla_{\V}^*\X^*\X\nabla_{\V} &= [\X^*,\nabla_{\V}^*]\nabla_{\V}\X - \nabla_{\V}^*\X^*[\X,\nabla_{\V}] \\
& = - [\X,\nabla_{\V}^*]\nabla_{\V}\X + \nabla_{\V}^*\X[\X,\nabla_{\V}] \\
& = \nabla_{\HH}^* \nabla_{\V} \X - \nabla_{\V}^* \X \nabla_{\HH} \\
& = (\nabla_{\HH}^* \nabla_{\V} - \nabla_{\V}^* \nabla_{\HH}) \X - \nabla_{\V}^* [\X,\nabla_{\HH}] \\
& = - (n-1) \X^2 - \nabla_{\V}^* R_{FM} \nabla_{\V} = (n-1)\X^*\X - \nabla_{\V}^* R_{FM} \nabla_{\V}
\end{split}
\]
where we used in the second equality $\X^*=-\X$, in the third \eqref{equation:nablah} and \eqref{equation:divergence-commutators}, in the fifth \eqref{equation:formule-chiante} and \eqref{equation:commutator-nablah}, and in the last $\X^*=-\X$ again.
\end{proof}

\subsection{Associated Pestov identity}

As a corollary of Theorem \ref{theorem:global-pestov-frame}, we now prove similar global Pestov identities on homogeneous fibrations sitting ``between'' the frame bundle and the unit tangent bundle of the Riemannian manifold. This corresponds to the choice of a certain group $G \leqslant \mathrm{SO}(n-1)$; it will lead to an identity similar to \eqref{equation:global-pestov-frame} which we call the \emph{associated Pestov identity}.

\subsubsection{Homogeneous fibrations over spheres} Let $G \leqslant \mathrm{SO}(n-1) \hookrightarrow \mathrm{SO}(n)$ be a subgroup of $\mathrm{SO}(n-1)$, seen as a subgroup of $\mathrm{SO}(n)$ via the diagonal embedding
\[
\mathrm{SO}(n-1)\ni a \mapsto \begin{pmatrix} 1 & 0 \\ 0 & a \end{pmatrix} \in \mathrm{SO}(n).
\]
and let $P := \mathrm{SO}(n)/G$. This defines a fibration over $S^{n-1} = \mathrm{SO}(n)/\mathrm{SO}(n-1)$ with typical fiber given by $F := \mathrm{SO}(n-1)/G$. More precisely, one has $\mathrm{SO}(n) \stackrel{\pi}{\to} P \stackrel{\pi'}{\to}S^{n-1}$, where $\pi,\pi'$ are both Riemannian submersions.

We consider again the standard basis $(\e_1,\dotsc,\e_n)$ of $\R^n$. The fibration $\pi' \circ \pi : \mathrm{SO}(n) \to S^{n-1}$ is given by the map $a \mapsto a\,\mathbf{e}_1$. The tangent bundle of $\mathrm{SO}(n)$ is trivial and can be naturally identified at each point $w \in \mathrm{SO}(n)$ with the Lie algebra $\mathfrak{so}(n) \simeq \Lambda^2 \R^n$. This further decomposes as
\begin{equation}
    \label{equation:further-decomp}
\Lambda^2 \R^n = \mathfrak{so}(n) = \R^{n-1} \oplus \mathfrak{so}(n-1) = \mathfrak{v}_0 \oplus \ker \dd_{\id}(\pi' \circ \pi),
\end{equation}
where we introduce the notation $\mathfrak{v}_0 :=\mathrm{Span}(\e_1 \wedge \e_j)_{2 \leq j \leq n}\simeq \R^{n-1}$ and $\mathfrak{so}(n - 1) = \mathrm{Span}(\e_i \wedge \e_j)_{2 \leq i < j \leq n}$. Observe that the restriction of the metric on $\Lambda^2\R^n$ to $\R^{n-1}$ is the standard metric making $(\e_j)_{2 \leq j \leq n}$ into an orthonormal basis.

Let $\mathfrak{v}_2 \leqslant \mathrm{so}(n-1)$ be the Lie algebra of $G$, and $\mathfrak{v}_1$ be its orthogonal in $\mathfrak{so}(n-1)$. Then, we can further decompose $\mathfrak{so}(n) = \mathfrak{v}_0 \oplus \mathfrak{v}_1 \oplus \mathfrak{v}_2$; correspondingly, the tangent bundle $T\mathrm{SO}(n)$ splits as $T\mathrm{SO}(n) = \V_{\mathrm{SO}(n)}^0 \oplus \V_{\mathrm{SO}(n)}^1 \oplus \V_{\mathrm{SO}(n)}^2$ where $\V_{\mathrm{SO}(n)}^2 := \ker \dd\pi$. Moreover, we can decompose $TP = \V^0_P \oplus \V^1_P$, where $\V^1_P = \ker \dd\pi'$ and $\V^0_P = \dd\pi(\V_{\mathrm{SO}(n)}^0)$. Overall, we thus obtain the following diagram:
\begin{equation}
\label{equation:diagram}
\vcenter{
\xymatrix@C=1em{
    T\mathrm{SO}(n) \ar[d]_{\dd\pi} & = & \V_{\mathrm{SO}(n)}^0 \ar[d]_{\dd\pi} & \oplus & \V_{\mathrm{SO}(n)}^1  \ar[d]_{\dd\pi} & \oplus & \V_{\mathrm{SO}(n)}^2  \ar[d]_{\dd\pi} \\ 
    TP \ar[d]_{\dd\pi'} & = & \V_P^0 \ar[d]_{\dd\pi'} & \oplus & \V_P^1  \ar[d]_{\dd\pi'}& \oplus  & \{0\} \ar[d]_{\dd\pi'} \\
    TS^{n-1} & = & \V^0_{S^{n-1}} & \oplus & \{0\} & \oplus & \{0\},
  }
  }
\end{equation}
where the maps on the left-hand side of the equalities are Riemannian submersions, and the maps on the right-hand side not mapping to $\{0\}$ are isometries.

\subsubsection{Associated Pestov identity} Let $G \leqslant \mathrm{SO}(n-1)$ as before. We now consider the associated bundle $PM := FM \times_{\rho} F$ over $SM$  with typical fiber $F := \mathrm{SO}(n-1)/G$, defined as the set of equivalence classes $[w,f]$ for the relation 
\[
    (wa^{-1},f) \sim (w,\rho(a)f),\quad (w,f) \in FM \times F,\ a \in \mathrm{SO}(n-1),
\]
where $\rho : \mathrm{SO}(n-1) \to \mathrm{Aut}(F)$ is the left-action of $\mathrm{SO}(n-1)$ on $F$. We remark that one can write $PM = FM/G$, and that $PM$ also fibers over $M$ as an associated bundle of $FM \to M$ with respect to the representation $\rho : \mathrm{SO}(n) \to \mathrm{Aut}(\mathrm{SO}(n)/G)$ by left multiplication, that is $PM = FM \times_\rho (\mathrm{SO}(n)/G)$.
Choosing a base point $o \in F$ allows to define a natural Riemannian submersion $\pi : FM \to PM$ by setting $\pi(w) := [w,o]$. In the particular case where $G = \mathrm{SO}(n-1)$ and $F = S^{n-1}$, one gets $PM = SM$, the unit sphere bundle over $M$. Since $G \leqslant \mathrm{SO}(n-1)$, we observe that $\pi' : PM \to SM$ is naturally a Riemannian submersion over $SM$ making the following diagram of Riemannian submersions commute
\[
  \xymatrix{ FM \ar[r]_{\pi} \ar[rd]_{\pi_{FM}} & PM \ar[r]_{\pi'} & SM \ar[ld]^{\pi_{SM}} \\ & M & }
\]
As $G \leqslant \mathrm{SO}(n-1) \hookrightarrow \mathrm{SO}(n)$, given $a \in G$, one has
\begin{equation}
\label{equation:be1}
(R_a)_* \X  = (R_a)_* B_{\e_1} =B_{a^{-1}\e_1} = B_{\e_1} = \X ,
\end{equation}
and thus $\X $ descends to a vector field $\X_{PM}$; if $PM = SM$, then $\X_{PM} = \X_{SM}$ is the usual geodesic vector field. The vector field $\X_{PM}$ generates a flow $(\Phi_t^{PM})_{t \in\R}$ which is an \emph{extension} of the geodesic flow $(\varphi_t)_{t \in \R}$ on $SM$ in the sense that
\[
\pi' \circ \Phi_t^{PM} = \varphi_t \circ \pi', \qquad \forall t \in \R.
\]
If $(M,g)$ has negative sectional curvature, $(\varphi_t)_{t \in \R}$ is a \emph{uniformly hyperbolic} flow (also called \emph{Anosov} in the literature), while $(\Phi_t^{PM})_{t \in \R}$ is a \emph{partially hyperbolic flow} if $G$ is strictly contained in $\mathrm{SO}(n-1)$.

\begin{example}
Taking $G = \mathrm{SO}(n-2) \hookrightarrow \mathrm{SO}(n-1)$ (via the diagonal embedding), one obtains as typical fiber $\mathrm{St}_2 := \mathrm{SO}(n)/\mathrm{SO}(n-2)$, the Stiefel space of orthogonal $2$-frames, and the associated bundle $PM = \mathrm{St}_2M$ is the Stiefel bundle of orthogonal $2$-frames over $M$. The induced flow $(\Phi_t^{\mathrm{St}_2M})_{t \in \R}$ is the $2$-frame flow.
\end{example}

By \eqref{equation:diagram}, the tangent bundles of $FM, PM$ and $SM$ split as
\begin{equation}
 \label{equation:split-tfm} 
 \vcenter{
\xymatrix@C=1em{
    T(FM) \ar[d]_{\dd\pi} & = & \HH_{FM} \ar[d]_{\dd\pi} & \oplus & \V_{FM}^0 \ar[d]_{\dd\pi} & \oplus & \V_{FM}^1  \ar[d]_{\dd\pi} & \oplus & \V_{FM}^2  \ar[d]_{\dd\pi} \\
  T(PM) \ar[d]_{\dd\pi'} & = & \HH_{PM} \ar[d]_{\dd\pi'} & \oplus & \V_{PM}^0 \ar[d]_{\dd\pi'} & \oplus & \V_{PM}^1  \ar[d]_{\dd\pi'}& \oplus  & \{0\} \ar[d]_{\dd\pi'} \\
    T(SM) & = & \HH_{SM} & \oplus & \V^0_{SM} & \oplus & \{0\} & \oplus & \{0\},
    }
    }
\end{equation}
where we added the subscript $FM,PM$ and $SM$ to keep track of the bundle. 

Define the vertical space on $PM$ as $\V_{PM} := \V_{PM}^0 \oplus \V_{PM}^1$. By \eqref{equation:split-tfm}, given $w \in FM$ and $T \in (\V^{i}_{PM})_{\pi(w)}$, there is a well-defined vertical lift $T_w^{\V^{FM}} \in (\V^{i}_{FM})_w$ for $i\in\{0,1\}$. Observe that if $T^{\V^{FM}}_w=Y_\xi$ then $T^{\V^{FM}}_{wa}=Y_{a^{-1}\xi}$ for every $a\in G$. Here $a^{-1}\xi$ denotes the adjoint action of $a^{-1}$ on $\xi \in \mathfrak{so}(n)$.

The \emph{vertical gradient} \begin{equation}
    \nabla_{\V}^{PM} : C^\infty(PM) \to C^\infty(PM,\V_{PM}),
\end{equation}
is defined as the orthogonal projection of the total gradient onto $\V_{PM}$, and further splits as $\nabla_{\V}^{PM} = \nabla_{\V^0}^{PM} + \nabla_{\V^1}^{PM}$.

\begin{lemma}
For all $u \in C^\infty(PM)$, 
\begin{equation}
\label{equation:nablav-inv} 
\nabla_{\V}^{PM} u(z) := \dd\pi\left( [\nabla_{\V}^{FM} \pi^* u]^{\V}(w) \right),\quad z \in PM, w \in \pi^{-1}(z).
\end{equation}
\end{lemma}

\begin{proof}
    Let $T\in \V_{PM}(z)$ be any vertical tangent vector. Since $\pi$ is a Riemannian submersion, we have
    \begin{eqnarray*}
    \langle \dd\pi\left( [\nabla_{\V}^{FM} \pi^* u]^{\V}(w) \right),T\rangle&=&
    \langle  [\nabla_{\V}^{FM} \pi^* u]^{\V}(w) ,T_w^{\V_{FM}}\rangle = T_w^{\V_{FM}}(\pi^* u)\\
    &=&\dd u(\dd\pi(T_w^{\V_{FM}}))=\dd u(T)=\langle \nabla_{\V}^{PM} u(z),T\rangle.
    \end{eqnarray*}
\end{proof}

Similarly to \eqref{equation:curvature}, we can introduce the curvature endomorphism on $PM$ as a section $R_{PM} \in C^\infty(PM,\mathrm{End}(\V_{PM}))$ given at $z \in PM$ by
\begin{equation}
\label{equation:r}
\langle R_{PM}(z)T, T' \rangle := \langle R_{FM}(w) (\xi(w)), \xi'(w) \rangle,
\end{equation}
where $w \in FM$ is any point such that $\pi(w)=z$ and $\xi(w),\xi'(w)\in\mathfrak{so}(n)$ are defined by $T^{\V_{FM}}_w=Y_{\xi(w)}$ and $(T')^{\V_{FM}}_w = Y_{\xi'(w)}$. We claim that $R_{PM}(z)$ is well-defined, i.e. independent of the choice of $w \in FM$ such that $\pi(w)=z$. Indeed, \eqref{equation:curvature} shows that for every $a\in\mathrm{SO}(n-1)$ and $w\in FM$, $R_{FM}(wa)=a^{-1}\circ R_{FM}(w)\circ a$, where the action of $a\in\mathrm{SO}(n-1)\subset \mathrm{SO}(n)$ on $\Lambda^2\R^n$ is the adjoint action. Moreover, for every $a\in G$ we have $\xi(wa)=a^{-1}\xi(w)$ and $\xi'(wa)=a^{-1}\xi'(w)$. Thus 
\begin{equation}
\label{equation:r-inv}
\langle R_{FM}(w) (\xi(w)), \xi'(w) \rangle=\langle R_{FM}(wa) (\xi(wa)), \xi'(wa) \rangle,\quad a \in G.
\end{equation}

Moreover, observe that by definition $R_{PM}$ maps into $\V_{PM}^0$ and therefore splits as a sum $R = R_0 + R_1$, where $R_0(z) \in \End(\V_{PM}^0)$ and $R_1 \in \mathrm{Hom}(\V_{PM}^1,\V_{PM}^0)$. We can further express the $R_0$ part. Indeed, recall that for $v \in SM$, the connection map $\mc{K} : \V^0_{SM} \to v^\perp$ was introduced in \S\ref{ssection:unit-tangent-bundle}. For every $z \in PM$, writing $v := \pi'(z) \in SM$, there is a natural identification $\mc{K} \dd\pi' : \V^0_{PM} \to v^\perp$. One can then further obtain:

\begin{lemma}
For all $z \in PM$, $T,T' \in \V_{PM}(z)$:
\begin{equation}
\label{equation:r-splitting}
\begin{split}
\langle R_{PM}(z)T, T' \rangle & = \langle R_0(z) T, T' \rangle + \langle R_1(z) T, T' \rangle \\
& = \langle \mc{R}_x\left(\mc{K} \dd\pi'(T) \wedge v\right), v \wedge \mc{K} \dd\pi'(T') \rangle + \langle R_1(z) T, T' \rangle,
\end{split}
\end{equation}
where $x:=\pi_{SM}(v)$. 
\end{lemma}

\begin{proof}By \eqref{equation:r}, it suffices to prove this equality on $PM = FM$. In this case, $\pi=\mathrm{Id}$ and $\pi' : FM \to SM$ is the standard projection $w \mapsto v = w(\e_1)$. 

Let $\xi(w),\xi'(w)\in\mathfrak{so}(n)$ be defined by $T_w=Y_{\xi(w)}$ and $T'_w = Y_{\xi'(w)}$. Since $\V_{FM}$ can be identified with $ \mathfrak{so}(n)=\R^{n-1}\oplus\mathfrak{so}(n-1)$ (see \eqref{equation:further-decomp}), one can decompose
\[
    \xi(w) =: \e_1 \wedge \theta + \xi_1,\quad \xi'(w) =: \e_1 \wedge \theta' + \xi_1',\quad \theta, \theta' \in \R^{n-1},\quad \xi_1, \xi_1' \in \mathfrak{so}(n-1),\]
    and correspondingly $T=T_0+T_1$, $T'=T'_0+T'_1$. We then compute
    \[\dd\pi'(w)(T_1) = \dfrac{\dd}{\dd t}|_{t=0}( w \circ e^{t\xi_1} \e_1)=0,
\]
and
\[
    \mc{K} \dd\pi'(w)(T_0) =\mc{K}  \dfrac{\dd}{\dd t}|_{t=0} (w \circ e^{t \e_1 \wedge \theta} \e_1) = w(\theta),\qquad \forall\theta \in \mathbb{R}^{n - 1}, w\in FM,
\]
thus showing that $\mc{K} \dd\pi'(w)(T)=w(\theta)$, and similarly $\mc{K} \dd\pi'(w)(T')=w(\theta')$.
    Then, using \eqref{equation:curvature}, we find:
\[
\begin{split}
\langle R_{FM}(w)T,T'\rangle &= \underbrace{\langle R_{FM}(w)(\e_1 \wedge \theta), \e_1 \wedge \theta' \rangle}_{=\langle R_0(w)T,T'\rangle} +  \underbrace{\langle R_{FM}(w)(\xi_1), \e_1 \wedge \theta' \rangle}_{=\langle R_1(w)T,T'\rangle}  \\
& = \langle w(\e_1 \wedge \theta), \mc{R}_x(w(\theta') \wedge v)\rangle + \langle R_1(w)T,T'\rangle \\
& = \langle \mc{R}_x(v \wedge \mc{K}\dd\pi'(T)),\mc{K}\dd\pi'(T') \wedge v \rangle + \langle R_1(w)T,T'\rangle \\
& = \langle \mc{R}_x(\mc{K}\dd\pi'(T)\wedge v),v \wedge \mc{K}\dd\pi'(T') \rangle + \langle R_1(w)T,T'\rangle.
\end{split}
\]
where we have used in the last equality the fact that $\mc{R}_x$ is symmetric on $2$-forms.
\end{proof}

Observe that $\langle R_0(z) T, T' \rangle$ is the sectional curvature of $\mc{K} \dd\pi'(T) \wedge v$ when $T=T'$.
In the particular case where $G=\mathrm{SO}(n-1)$, one gets $PM=SM$, and the term $R_1$ in \eqref{equation:r-splitting} vanishes. Hence for all $v \in SM,$ and $T,T' \in (\V^0_{SM})_v$, 
\begin{equation}
\label{equation:courbure-sphere}
\langle R_{SM} T,T'\rangle =  \langle \mc{R}_x(\mc{K}T \wedge v), v \wedge \mc{K}T'\rangle.
\end{equation}

We can now derive the following associated Pestov identity:

\begin{corollary}[Associated Pestov identity]
\label{corollary:associated}
For all $u \in C^\infty(PM)$:
\begin{equation}
\label{equation:global-pestov-associated}
\begin{split}
 \|\nabla_{\V}^{PM} \X_{PM} u\|^2_{L^2(PM, \V_{PM})} - \|\X_{PM} \nabla_{\V}^{PM} u\|^2_{L^2(PM, \V_{PM})}&\\ 
 &\hspace{-7cm}= (n-1)\|\X_{PM} u\|^2_{L^2(PM)} - \langle R_{PM} \nabla_{\V}^{PM} u, \nabla_{\V}^{PM} u \rangle_{L^2(PM, \V_{PM})}.
\end{split}
\end{equation}
\end{corollary}
Note that by a slight abuse of notation, $\X_{PM}$ in the second term of the left hand side of \eqref{equation:global-pestov-associated} denotes the first order differential operator $\nabla^{PM}_{\X_{PM}}$ on $C^\infty(PM, \V_{PM})$, where $\nabla^{PM}$ is the covariant derivative on $\V_{PM}$ determined by the Levi-Civita connection on $FM$.
\begin{proof}
For $u \in C^\infty(PM)$, we apply the Pestov identity \eqref{equation:global-pestov-frame} to $\pi^* u \in C^\infty(FM)$. Combining \eqref{equation:nablav-inv} and \eqref{equation:r-inv} ($G$-equivariance of the vertical gradient and the curvature operator), we get that the functions appearing in the Pestov identity for \eqref{equation:global-pestov-frame} for $\pi^*u$ are $G$-invariant. Hence, using the Fubini Theorem, we obtain \eqref{equation:global-pestov-associated}.
\end{proof}

\begin{example}[Pestov identity on the unit tangent bundle] As a particular example, taking $G=\mathrm{SO}(n-1)$, we have $PM=SM$ and one retrieves the usual global Pestov identity on the unit tangent bundle (see \cite[Proposition 2.2]{Paternain-Salo-Uhlmann-15} for instance): for all $u \in C^\infty(SM)$,
\[
 \|\nabla_{\V}^{SM} \X_{SM} u\|^2 - \|\X_{SM}\nabla_{\V}^{SM} u\|^2 = (n-1)\|\X_{SM} u\|^2- \langle R_{SM} \nabla_{\V}^{SM} u, \nabla_{\V}^{SM} u \rangle,
\]
where the norms are computed in ${L^2(SM, \V_{SM})}$, except for the first term on the right hand side, which is computed in ${L^2(SM)}$.
Note that the curvature term was computed in \eqref{equation:courbure-sphere} and only involves the sectional curvature. When specialized to a function $u \in C^\infty(SM)$ that is a spherical harmonic of degree $m \geq 0$ in every spherical fiber of $SM$ (that is, $\Delta_{\V}^{SM} u = m(m+n-2)u$), one obtains (after non-trivial simplifications) the following formula known as the \emph{localized Pestov identity}:
\[
\dfrac{(n+m-2)(n+2m-4)}{n+m-3}\|\X_-u\|^2 - \dfrac{m(n+2m)}{m+1}\|\X_+u\|^2+\|Z_m u\|^2 = \langle R_{SM}\nabla_{\V} u,\nabla_{\V} u\rangle,
\]
where $Z_m$ is a differential operator of order $1$, and $\X_{SM}=\X_-+\X_+$ is the decomposition into raising and lowering operators, see \cite[Theorem 14.3.4]{Lefeuvre-book} for instance.
\end{example}

\begin{remark}Unfortunately, the unit tangent bundle $SM$ is the only associated bundle for which the curvature term introduced in \eqref{equation:r} and appearing on the right-hand side of the Pestov identity \eqref{equation:global-pestov-associated} has a sign in negative sectional curvature. Nevertheless, this curvature term does have a sign also for other associated bundles $PM$, provided that the metric is sufficiently pinched.
\end{remark}

\section{Applications in negative curvature: frame flow ergodicity}

\label{section:applications}

We now apply the previous Pestov identities in order to study the \emph{frame flow ergodicity} over a negatively curved Riemannian manifold $(M,g)$. The frame flow \eqref{equation:frame-flow} preserves the Liouville measure $\mu$, and ergodicity (with respect to $\mu$) is the property that
\[
L^2(FM,\mu) \cap \ker \X  = \C \cdot \mathbf{1}_{FM},
\]
where $\mathbf{1}_{FM}$ is the constant function on $FM$ equal to $1$ everywhere. On odd-dimensional manifolds (and dimension $n \neq 7$), frame flow ergodicity was proved by Brin-Gromov \cite{Brin-Gromov-80} using topological arguments. In even dimensions, Kähler manifolds of negative sectional curvature are natural counter-examples to frame flow ergodicity, and it is thus conjectured that, unless $(M,g)$ has some special holonomy reduction, its frame flow should be ergodic (see \cite[Conjecture 2.9]{Brin-82}). 

The \emph{pinching} of a negatively curved Riemannian manifold (with sectional curvature $\kappa_g$ normalized so that it is $\geq -1$) is the largest number $\delta \in (0,1)$ such that
\[
-1 \leq \kappa_g \leq -\delta.
\]
Riemannian manifolds with special holonomy satisfy $\delta \leq 0.25$ and it is thus natural to study the weaker conjecture of frame flow ergodicity for $0.25$-pinched Riemannian manifolds (that is $\delta > 0.25$). In \cite{Brin-Karcher-84,Burns-Pollicott-03}, it was proven that on even-dimensional manifolds (and in dimension $7$), the frame flow is ergodic under \emph{some pinching condition} on the curvature very close to $1$. This was improved by the authors in \cite{Cekic-Lefeuvre-Moroianu-Semmelmann-24} to get $\delta \sim0.27$ in dimensions $\dim M \equiv_4 2$. We also point out that similar ergodicity results hold in the setting of unitary frame flows over Kähler manifolds with negative sectional curvature, see \cite{Brin-Gromov-80} for complex odd-dimensional manifolds and \cite{Cekic-Lefeuvre-Moroianu-Semmelmann-24-2} for the even-dimensional case.

The (non-)ergodicity of the frame flow is described by means of a subgroup $H \leqslant \mathrm{SO}(n-1)$ called the \emph{transitivity group}, see \cite{Brin-75-1,Brin-75-2,Lefeuvre-21}:

\begin{theorem}[Characterization of ergodicity, \cite{Lefeuvre-21}]
\label{theorem:lefeuvre}
There exists a natural identification
\[
\Psi : L^2(H\backslash\mathrm{SO}(n-1)) \rightarrow L^2(FM,\mu) \cap \ker \X .
\]
In particular, the frame flow $(\Phi_t)_{t \in \R}$ is ergodic if and only if $H = \mathrm{SO}(n-1)$. Moreover, if $f \in C^\infty(H\backslash\mathrm{SO}(n-1))$, then the corresponding flow-invariant function $\Psi f \in \ker \X $ is also smooth, that is, $\Psi f \in C^\infty(FM)$. 
\end{theorem}

\begin{example}[Frame flow ergodicity on hyperbolic manifolds]
It follows immediately from the global Pestov identity \eqref{equation:global-pestov-frame}, Example \ref{example:curvature} and Theorem \ref{theorem:lefeuvre} that the frame flow is ergodic on hyperbolic manifolds. Indeed, assume that it is not ergodic; then there exists a smooth non-constant function $u \in C^\infty(FM) \cap \ker \X $. Applying \eqref{equation:global-pestov-frame} together with Example \ref{example:curvature}, we obtain
\[
-\|\X\nabla_{\V}u\|^2_{L^2(FM, \Lambda^2 \mathbb{R}^n)} = -\langle R_{FM}\nabla_{\V}u,\nabla_{\V}u\rangle_{L^2(FM)} = \|\nabla_{\V}u \cdot\e_1\|^2_{L^2(FM)}.
\]
Hence, $\nabla_{\V}u \cdot \e_1 = 0 = \X \nabla_{\V} u$. By \eqref{equation:nablah}, we also get $\nabla_{\HH} u = 0$, that is $B_{\e_j}u = 0$ for $2 \leq j \leq n$. Since $0=\X u= B_{\e_1}u$, this yields $B_{\e_j} u = 0$ for all $1 \leq j \leq n$. Moreover, observe that by \eqref{equation:commutateur}, 
\[
[B_{\e_j},B_{\e_{j'}}] = - [w^{-1}\mc{R}_{\pi_{FM}(w)}(w(\e_j) \wedge w(\e_{j'}))]^{\V} = - (\e_j \wedge \e_{j'})^{\V},
\]
and we thus get $ (\e_j \wedge \e_{j'})^{\V} u = 0$ for all $1 \leq j,j' \leq n$. 

Since the vector fields $(\e_j \wedge \e_{j'})=(Y_{ij})_{i < j}$ span all the vertical directions, and the vector fields $(B_{\e_j})_{1 \leq j \leq n}$ span all the horizontal directions, we deduce that $\dd u = 0$, that is, $u$ is constant. More generally, this argument still works for metrics with $\delta$-pinched sectional curvature for $\delta$ close enough to $1$.
\end{example}

\bibliographystyle{alpha}
\bibliography{Biblio}

\end{document}